\documentclass[12pt]{article}

\usepackage[centertags]{amsmath}
\usepackage{amsfonts}
\usepackage{amssymb}
\usepackage{amsthm}
\usepackage{newlfont}

\newlength{\defbaselineskip}
\newcommand{\setlinespacing}[1]%
           {\setlength{\baselineskip}{#1 \defbaselineskip}}

\theoremstyle{plain}
\newtheorem{thm}{Theorem}[section]
\newtheorem{cor}[thm]{Corollary}
\newtheorem{lem}[thm]{Lemma}

\theoremstyle{definition}

\newtheorem{rem}[thm]{Remark}
\newtheorem{note}[thm]{Note}
\numberwithin{equation}{section}

\begin{document}

\newcommand{\ol }{\overline}
\newcommand{\ul }{\underline }
\newcommand{\ra }{\rightarrow }
\newcommand{\lra }{\longrightarrow }
\newcommand{\ga }{\gamma }
\newcommand{\st }{\stackrel }
\newcommand{\scr }{\scriptsize }

\title{\Large\textbf{A Theorem of Van Kampen Type for Pseudo Peano Continuum Spaces}}
\author{\textbf{Hanieh Mirebrahimi\footnote{Correspondence: h$_{-}$mirebrahimi@um.ac.ir} and Behrooz Mashayekhy\footnote{  This research was supported by a grant from Ferdowsi  University of Mashhad; (No. MP87151MSH). bmashf@um.ac.ir}}\\
Department of Pure Mathematics,\\ Center of Excellence in Analysis on Algebraic Structures,\\
Ferdowsi University of Mashhad,\\
P. O. Box 1159-91775, Mashhad, Iran.}
\date{ }
\maketitle
\begin{abstract}
G. Conner and K. Eda (Topology and its Applications, 146, (2005),
317-328.) introduced a new construction of spaces from groups.
They remarked that the construction is not categorical. In this
paper, based on the work of Conner and Eda, we construct two new
categories for which the functor of the fundamental group $\pi$
has a right adjoint and consequently is right exact and preseves
direct limits. Also, we study the behavior of the functor $\pi$ on
quotient spaces and give a new version of Van Kampen theorem for
join of spaces in the new categories.
\end{abstract}
\section{Introduction and Motivation}
K. Eda [2] proved that if $X$ and $Y$ are one-dimensional, locally
path connected, path connected, metric spaces which are not
semilocally simply connected at any point and have isomorphic
fundamental groups, then $X$ and $Y$ are homeomorphic.

G. Conner and K. Eda [1] based on the above fact proposed a new
construction of a space from a group and showed that this
construction arises from the fundamental group of a space of the
above type and it is homeomorphic to the topological space itself.
They also remarked that the construction of topological spaces
from groups is not categorical [1, Remark 2.8 (2)]. They called a
point $x\in X$ a \emph{wild point}, if $X$ is not semilocally
simply connected at $x$ and denoted $X^{w}$ for the subspace
consisting of all wild points of $X$. They also called a space
$X$ \emph{wild} if $X=X^{w}$.

In this paper, we call, for abbreviation, a one-dimensional,
locally path connected, path connected, metric space as a \emph{
pseudo Peano continuum } and study some interesting properties of
these spaces in section three in order to find some functorial
properties related to the Conner-Eda constructions. In section
four, we construct two new categories ${\mathbf{sqTop^w}}$ and
${\pi(\mathbf{Top^w})}$ for which the functor $\bar{\pi}$,
introduced by the fundamental group $\pi$, has a right adjoint and
consequently is right exact and preserves direct limits in the new
categories. These properties enable us to present in section five
a new version of Van Kampen theorem for joins of wild pseudo Peano
continuum spaces. Finally, in section six, using the right
exactness of $\bar{\pi}$, we study the behavior of the
fundamental group on quotient spaces of wild
pseudo Peano continua.
\section{Definitions and Preliminaries}
We are supposed that the reader is familiar with basic concepts
and properties in group theory, topology, algebraic topology and
category theory. In this section, we recall some essential facts
from [1,2,3]. Firstly, given two categories $\mathcal{C}$ and
$\mathcal{D}$, a pair of functors $(T_1,T_2)$, where
$T_1:\mathcal{C}\rightarrow \mathcal{D}$ and
$T_2:\mathcal{D}\rightarrow \mathcal{C}$, is said to be
\emph{adjoint} if there exists a natural equivalence
\[{\mathrm Hom}_{\mathcal{D}}(T_1 X,Y)\simeq {\mathrm Hom}_{\mathcal{C}}(X,T_2 Y),\]
for all $X\in \mathcal{C}$ and $Y\in \mathcal{D}$. The functor
${T_2}$ is called the \emph{right-adjoint functor} to $T_1$
and the pair ($T_1,T_2$) is called an \emph{adjoint pair}.

It is well-known that every functor which has right-adjoint
preserves direct limits [7], that is for any direct systems $\{X_i;
\lambda_i^j,I\}$ indexed by a partially ordered set I, in the
category $\mathcal{C}$, we have
\[T_1(\lim_{\longrightarrow} X_i)=\lim_{\longrightarrow} T_1(X_i).\]

In addition, we recall two important topological structures, join
and weak join [4,7] of a family of pointed topological spaces
$\{(X_i,x_i)\}$ indexed by a partially ordered set I. The \emph{join}
space $(X,x)=\bigvee_{i\in I}(X_i,x_i)$ of the $X_i$ at the
origin $x$ is the quotient of the disjoint union of the $X_i$ by
the relation which identifies all the copies $x_i$ of $x$. The
\emph{weak join} of these spaces, denoted by
$(\tilde{X},x)=\tilde{\bigvee}_{i\in I}(X_i,x_i)$, is defined as a
different topology on the same underlying set, that of $\tilde{X}$
is coarser. The topology of $\tilde{X}$ is exactly the same
everywhere except at the distinguished point $x$. A neighborhood
of $x$ in $X$ is any set of the form $\bigvee_{i\in I}U_i$, where
$U_i$ is a neighborhood of $x_i$ in $X_i$. A neighborhood of $x$
in $\tilde{X}$ has the same form but must also satisfy $U_i=X_i$
for almost all $i$.

Note that the above constructions, join and weak join spaces,
 can be considered as direct limit and inverse limit
of special systems in the category of pointed topological spaces,
respectively, [8]. For the structure of direct limit and inverse
limit in the category ${Top}_*$, we only recall that direct limit
is a particular quotient of join space and inverse limit is a
special subspace of product space. For further details we refer
the reader to [8].

We also recall some topological definitions. A topological space
$X$ has \emph{topological dimension} $m$, if every covering
$\cal U$ of $X$ has a refinement $\cal U'$ in which every point
of $X$ occurs in at most $m+1$ sets in $\cal U'$, and $m$ is the
smallest such integer [5].

A locally path connected, path connected, compact metric space is
called \emph{Peano continuum} [1]. A space $X$ is said to be
\emph{semilocally simply connected} (\emph{semilocally
$1$-connected}) [6] at $x\in X$ if there exists an open
neighborhood $U_x$ of $x$ so that the inclusion map
$i:U_x\hookrightarrow X$ induces the trivial homomorphism
$i_*:\pi(U_x,x)\hookrightarrow \pi(X,x)$, otherwise the point
$x\in X$ is called a wild point at $X$.

We continue this section by pointing out some algebraic and
topological concepts that have been studied in [1,3].

Let $\{G_i\ |\ i\in X\}$ be a family of groups, $\prod^{*}_{i\in X}G_i$
be the free product of the $G_i$ for $X\subset I$, and
$p_{XY}:\prod^{*}_{i\in Y}G_i\rightarrow \prod^{*}_{i\in X}G_i$ be the canonical
homomorphism for $X\subset Y\subset I$. Then the unrestricted free
product of the family is the inverse limit
$\displaystyle{\lim_{\leftarrow}(\prod^{*}_{i\in X}G_i,p_{XY}:X\subset
Y\sqsubset I)}$, where $Y\sqsubset I$ means that $Y$ is the finite
subset of $I$ [3]. The free products of groups are defined by
using words of finite length, and the infinitary version of the
free products will be defined by using words of infinite length
instead of finite one. Indeed our interest is concentrated on
\emph{free $\sigma$-product } of $G_i$ $(i\in I)$, denoted by
$X_{i\in I}^{\sigma}G_i$, which is defined by using words of
countable length and is a subgroup of the unrestricted free
product [3].

Note that one of the most interesting theorems which we use in
this paper asserts that the fundamental group of a weak join of a
family of some topological spaces is isomorphic to the free
$\sigma$-product of their fundamental groups [3, Theorem A.1]. For
instance, we point out the well-known \emph{Hawaiian earring
space}, is a weak join of countably many circles, whose
fundamental group is isomorphic to the free $\sigma$-product
$X^{\sigma}_{n<\omega}{\mathbb{Z}_n}$, where ${\mathbb Z}_n$,
for $n<\omega$, is a copy of the infinite cyclic group ${\mathbb
Z}$.

Now we mention one of the most essential structures, introduced by
 Conner and Eda in [1] and our results are mainly based on it. For
an arbitrary group $G$, ${\cal H}_G$ is defined to be the set of
all subgroups of $G$ which are homomorphic images of
$X^{\sigma}_{n<\omega}{\mathbb{Z}_n}$. A finite subset $F$ of
${\cal H}_G$ is said to be \emph{compatible} if there exists
$H\in {\cal H}_G$ such that $\cup F\subseteq H$. A subfamily $C$
of ${\cal H}_G$ is said to be \emph{compatible} if any finite
subset of $C$ is compatible. Let ${\cal X}_G$ be the set of all
maximal compatible subfamilies of ${\cal H}_G$ which contain an
uncountable subgroup. For subgroups $H$ and $H'$ of $G$, we denote
$H\preceq H'$ if there exists $F\sqsubset G$ such that $H\leq
\langle H'\cup F\rangle$.

Also, a topology on the set ${\cal X}_G$ is introduced so that for
any $Y\subseteq {\cal X}_G$, $Y$ is closed if contains all its
limit points. Note that a point $x$ is called to be a limit point
of $Y$ if there exists a sequence $(x_n:n<\omega)$ of elements of
$Y$ satisfying the condition that for given uncountable $H_n\in
x_n (n<\omega)$ there exists $H'_n\in x_n$ such that:

- $H_n\preceq H'_n$;

- for arbitrary $a_n\in H'_n (n<\omega)$ there exists
$h:X^{\sigma}_{n<\omega}{\mathbb{Z}}_n\rightarrow G$ such
that $h(\delta_n)=a_n$ for every $n<\omega$ and
${\mathrm{Im}}(h)\in x$, where ${\mathbb Z}_n$ is a copy of the
infinite cyclic group ${\mathbb Z}$ and $\delta_n$ it's generator
for any $n<\omega$.

Note that one of the goal, studied by Conner and Eda [1], is to
find a structure for the space ${\cal X}_G$ with respect to the
structure of $G$. For instance, we present some of these
interesting results which are the most useful in our main results
in this paper. In each point followed up, we refer to [1] for the
proof and further details.
\begin{thm} ([1, Theorem 2.2]).
Let $A$ be an abelian group. Then ${\cal X}_A$ is an empty or
one-point space. Specially ${\cal X}_{\mathbb Z}=\emptyset$.
\end{thm}
\begin{thm} ([1, Theorem 5.1]).
Let $X$ be a locally path connected, path connected,
one-dimensional metric space and $G$ be the fundamental group
$\pi(X)$. Then ${\cal X}_G$ is homeomorphic to $X^w$.
Consequently, if $X$ is wild, then ${\cal X}_G$ is homeomorphic to
$X$ itself.
\end{thm}
\begin{thm} ([1, Corollary 5.2]).
Let $X_n$ be a locally path connected, path connected,
one-dimensional metric spaces such that $X^w_n\neq \emptyset$,
for all $n<\omega$, and $G$ be the fundamental group
$\pi(\prod_{n<\omega}X_n)\simeq \prod_{n<\omega}\pi(X_n)$. Then
${\cal X}_G$ is homeomorphic to $\prod_{n<\omega}X^{\omega}_n$.
Consequently, if $X_n$ is wild, for every $n< \omega$, then ${\cal
X}_G$ is homeomorphic to $\prod_{n<\omega}X_n$.
\end{thm}
Finally, a group $S$ is said to be \emph{$n$-slender} [1] if and
only if for each homomorphism $h:X_{n<\omega} {\mathbb{Z}}_n
\rightarrow S$, the set $\{n<\omega :\ h(\delta_n)\neq e \}$ is
finite, where ${\mathbb Z}_n$ is a copy of the infinite cyclic
group ${\mathbb Z}$ and $\delta_n$ it's generator for any
$n<\omega$. The class $\mathcal{S}$ consists of all the groups
$G$ such that for any non-trivial element $g\in G$, there exists
an $n$-slender group $S$ and a homomorphism $h:G \rightarrow S$
with $h(g)\neq e$. Of course, in group theory, we can call
$\mathcal{S}$ as the class of all \emph{residually $n$-slender}
groups.
\begin{thm} ([1, Corollary 4.7]).
Let $A$ and $B$ be groups in $\mathcal{S}$. Then ${\cal
X}_{A*B}$ is the topological sum of ${\cal X}_A$ and ${\cal X}_B$.
\end{thm}
\begin{rem} First we note that
the class $\mathcal{S}$ of groups is closed under free product.
For a free product $G=\prod^{*}_{i\in I}G_i$ of groups $G_i$ in
$\mathcal{S}$ and $g\in G$, we consider a sequence
$g=g_{i_1}g_{i_2}\cdots g_{i_n}$ with $g_{i_j}\in G_{i_j}$ and
their corresponding homomorphisms $h_i:G_i\rightarrow S_i$ with
$h_i(g_i)\neq e$. The homomorphism $h:\prod^{*}_{i\in I}G_i\rightarrow
\prod^{*}_{i\in J}S_i$ clearly satisfies $h(g)\neq e$, since
\[h(g_i)=\left \{\begin{array}{ll} h_i(g_i)\ & \ if\ i\in J\\
\ e & \ otherwise.
\end{array}\right. \]
Hence $G$ belongs to $\mathcal{S}$. Now, using Theorem 2.4, we
conclude that for any family $\{G_i\}_{i\in I}$ of groups in
$\mathcal{S}$ and any finite partition $\{J_1,\cdots,J_n\}$ of
the index set $I$, the space ${\mathcal{X}}_{\prod^{*}_{i\in I}G_i}$ is
the topological sum of the spaces ${\mathcal{X}}_{*_{i\in
J_1}G_i}$, \ldots, ${\mathcal{X}}_{\prod^{*}_{i\in J_n}G_i}$.
\end{rem}
\begin{cor} For any free
product $A=\prod^{*}_{i\in I}A_i$ of finitely many abelian groups which
belong to the class $\mathcal{S}$, the space ${\cal X}_G$ is
empty or a finite space with the discrete topology. In
particular, for any free group $F$ of finite rank, the space
${\cal X}_F$ is empty.
\end{cor}
\section{Pseudo Peano Continuum Spaces}
In foregoing, we are dealing with pseudo Peano continua
introduced in section one.
\begin{note} For any pseudo Peano
continuum space X with $|X^w|\geq 2$, the fundamental group
$\pi(X)$ is not abelian. Indeed, using Theorem 2.1, we know that
for any abelian group $A$, ${\cal X}_A$ (and so the wild subspace
of a pseudo Peano continuum space whose fundamental group
 is isomorphic to $A$) is empty or
one-point which implies the result.
\end{note}
Now, we are interested in analyzing the wild subspaces of joins
and weak joins of pseudo Peano continuum spaces, which are
important in our categorical arguments.
\begin{lem} Let $\{(X_i,x_i)\}_{i\in I}$ be a finite family of pointed topological
spaces which are pseudo Peano continuum. Then the join of these
spaces, $X$ say, will be also a pseudo Peano continuum space.
\end{lem}
\begin{proof} Let $x_0\in X$ be the common point
of all $X_i$, so the path connectivity of all $X_i$ implies the
existence of paths from any point $x_0\neq x\in X$, belonging to a
unique space $X_{i_x}$, to the base point $x_0$. Hence the path
connectivity of the space $X$ is satisfied.

The  locally path connectivity of the space $X$ at any point
$x_0\neq x\in X$ is a direct result of the definition of join
structure and locally path-connectivity of the $X_i$. In fact, for
any neighborhood $N_{x_0}$ of the special point $x_0\in X$, for
any $i\in I$, the neighborhood $(U_{x_0})_i=N_{x_0}\cap X_i$
should be open in the corresponding space $X_i$. Hence, by the
locally path connectivity of any $X_i$, there exists the
path connected neighborhood $(V_{x_0})_i\subseteq (U_{x_0})_i$. So
the neighborhood $V=\cup_{i\in I} (V_{x_0})_i$ of $x_0$ as a
subspace of $N_{x_0}$ is our desirable neighborhood.

To show the metrizablity of the space $X$, firstly we consider the
corresponding metric $d_i$ to $X_i$, for any $i\in I$. Now using
these metrics, we define the following metric on the whole space
$X$, so that for any two points $x, y\in X$ which belong to $X_i$
and $X_j$, respectively, we have

\[d(x,y)=\left \{\begin{array}{ll} d_i(x,y)\ & {\mathrm if}\ i=j\\
 d_i(x,x_0)+ d_j(x_0,y)\ \  & \ \ \ \ {\mathrm otherwise}.
\end{array}\right. \]

We note that the well-definition of the metric $d$ is concluded by
the condition of being mutually disjoint for all $X_i$ in the
join space $X$. The others metric properties of $d$ are easily
deduced by consideration of $d_i$'s to be metrics.

Finally, the join space $X$ is obviously one-dimensional. It is
sufficient to note that the space $X$ as a disjoint union of
one-dimensional spaces is also one-dimensional. Indeed, for any
covering $\{U_{\lambda}\}_{\lambda\in \Lambda}$ of $X$ and any
$i\in I$, we consider ${\cal U}_i=\{U_{\lambda}\cap
X_i\}_{\lambda\in \Lambda}$ which is a cover for the
one-dimensional space $X_i$ and so it has a suitable refinement
$\tilde {\cal U}_i$, say. Hence the union $\cup_{i\in I}\tilde
{\cal U}_i$ is a cover of $X$ satisfying the condition of being
one-dimensional for the join space $X$.
\end{proof}
\begin{lem} Let
$\{(X_i,x_i)\}_{i\in I}$ be a countable family of pointed
topological spaces which are pseudo Peano continuum. Then the
weak join of these spaces, $(\tilde{X},x_0)$ say, will be also a
pseudo Peano continuum space.
\end{lem}
\begin{proof} Firstly, we recall that the
topology of two spaces, join and weak join of all $X_i$, are
exactly similar except at the base point $x_0$. In order to prove
the details, it is sufficient to investigate at the special point
$x_0$. Also we note that the path and the metric introduced in
the proof of the lemma 3.2, with the similar argument, are those
which show the path connectivity and the metrizability of the
weak join space $\tilde{X}$.

For the locally path connectivity of $\tilde{X}$ we consider the
similar neighborhood of $x_0$ and note that there exists $i_0$ so
that the neighborhood $N_{x_0}$ contains the whole space $X_j$,
for any $j\geq i_0$.

The one-dimensionality of $\tilde{X}$ is clearly satisfied. For
any covering $\{U_{\lambda}\}_{\lambda\in \Lambda}$ of
$\tilde{X}$, we consider the neighborhood $U_{\lambda_0}$ of
$x_0$ which, by the definition, contains all but a finite number
of $X_i$'s, $X_1, X_2, \cdots X_n$ say. Now we can consider
$\tilde{X}$ as the quotient space of the disjoint union
$(U_{\lambda_0}\cap \tilde{X})\cup X_1\cup \cdots \cup X_n$ that
for any $i$, $1\leq i\leq n$, ${\cal U}_i=\{U_{\lambda}\cap
X_i\}_{\lambda\in \Lambda}$ is a cover for the one-dimensional
space $X_i$ and so it has a suitable refinement $\tilde {\cal
U}_i$. Finally the union $(U_{\lambda_0}\cap \tilde{X})\cup\tilde
{\cal U}_1\cup \cdots \cup \tilde {\cal U}_n$ is a cover of
$\tilde{X}$ which obviously satisfied the condition of being
one-dimensional for the weak join space $\tilde{X}$ and this
completes the proof.
\end{proof}
At the end of this section, we summarize some important results of
[2] in our new language which assert precisely that for any two
wild pseudo Peano continuum spaces X and Y, the following
conditions are equivalent:\\
$(i)$ X , Y are homeomorphic spaces.\\
$(ii)$ X , Y have the same homotopy type.\\
$(iii)$ The fundamental groups of spaces X , Y are isomorphic.
\section{A Categorical Viewpoint}
In this section, we are ready to mention a categorical viewpoint
for the main concept of the paper i.e. pseudo Peano continuum
space. Specially we try to get some functorial properties of
fundamental groups which do not hold in general, but can be
satisfied in the particular category which we are going to define.

\hspace{-0.65cm}\textbf{{\bf The Category $\mathbf{Top^w}$;} }
which is a subcategory of $\mathbf{Top}_*$, consist of all spaces
which are wild pseudo Peano continuum and of all morphisms between
them.

\hspace{-0.65cm}\textbf{{\bf The Category $\mathbf{\pi(Top^w)}$;}}
which is a subcategory of the category $\mathbf{Group}$, consists
of all fundamental groups of wild pseudo Peano continuum spaces
and of all homomorphisms induced by the map between them.

By considering a congruence relation $\sim$ on the set of all
morphisms in the category $\mathbf{Top^w}$ such that $f\sim g$ if
and only if the induced homomorphism $\pi(f)=f_*$ and
$\pi(g)=g_*$ are equal in the category $\mathbf{Group}$, we can
construct the quotient category $\mathbf{qTop^w}$ and it tends to
define two functors $\bar{\pi}$ and ${\cal X}_{-}$ as follows. As
some famous objects of the category $\mathbf{qTop^w}$, we name
the spaces Menger sponge, Sierpinski gasket, Sierpinski carpet.

Note that for our purpose, we need to consider some suitable
subcategories of $\mathbf{qTop^w}$ as follows:

By a suitable subcategory of $\mathbf{qTop^w}$, denoted by
$\mathbf{sqTop^w}$, we mean all spaces in $\mathbf{qTop^w}$ with a
specific point which we correspond to each object with the
following manner.

Let $X$ be a wild pseudo Peano continuum space and $G=\pi(X)$ be
its fundamental group. By Theorem 2.2 there exists a
homeomorphism $\varphi:{\cal X}_{G}\rightarrow X$. Choose any
fixed point $*_G$ in ${\cal X}_{G}$ and consider $\varphi(*_G)$ as
the specific corresponding point for the space $X$. Now the
pointed space $({\cal X}_{G}, *_G)$ is one of the objects of the
suitable subcategory $\mathbf{sqTop^w}$. Note that for any
pointed space $(X,x_0)$ in $\mathbf{qTop^w}$, we can always
choose a suitable fixed point $*_{\pi(X)}$ in ${\cal X}_{\pi(X)}$
such that $\varphi(*_{\pi(X)})=x_0$.

\hspace{-0.65cm}\textbf{{\bf The functor $\mathbf{\bar{\pi}}$;}}
from a suitable subcategory of $\mathbf{qTop^w}$,
$\mathbf{sqTop^w}$ to $\mathbf{\pi({Top^w})}$ such that for any
space $X$ and any morphism $[f]$ in the category
$\mathbf{sqTop^w}$, defines $\bar{\pi}(X)=\pi(X)$ and
$\bar{\pi}([f])=f_*$.

\hspace{-0.65cm}\textbf{{\bf The functor $\mathbf{\cal X}_{-}$;}}
from the category $\pi(\mathbf{{Top^w}})$ to a category
$\mathbf{sqTop^w}$ such that for any group $G$ and any morphism
$f_*$ in $\pi(\mathbf{Top^w})$, defines ${\cal X}_{-}(G)= ({\cal
X}_{G},*_G)$ and ${\cal X}_{-}(f_*)={\cal X}_{f_*}=[f]$.

Now with respect to the above notation, we state the following
result.
\begin{thm} The pair of
functors ($\bar{\pi}$ , ${\cal X}_{-}$) is an adjoint pair.
\end{thm}
\begin{proof} First, we note that the functors
$\bar{\pi}$ and ${\cal X}_{-}$ are inverse to each other. Indeed,
using Theorem 2.2, for any $(X,x_0)\in {\mathbf{sqTop^w}}$, the
space ${\cal X}_{\pi(X)}$ is homeomorphic to the wild subspace of
$X$ and so by definition of $\mathbf{sqTop^w}$, the pointed space
${\cal X}_{-}\circ\bar{\pi}(X)=({\cal X}_{\pi(X)},*_{\pi(X)})$ is
equivalent to the space $(X,x_0)$ itself. Conversely, for any
$G\in \pi(\mathbf{Top^w})$, the claim $\bar{\pi}\circ{\cal
X}_{-}(G)\cong G$ is clearly satisfied.

Finally, for every space $(X,x_0)\in \mathbf{sqTop^w}$ and every group
$G=\pi({\cal X}_{G},*_G)\in \pi(\mathbf{Top^w})$, we define
$$\theta(=\theta_{X,G}): {\mathrm Hom}_{\pi(\mathbf{Top^w})}(\bar{\pi}(X),G)
\rightarrow {\mathrm Hom}_{\mathbf{sqTop^w}}((X,x_0),
({\cal X}_{G},*_G))$$
mapping $f_*$ to ${\cal X}_{f_*}=[f]$, where $f$ is a morphism in
${\mathrm Hom}_{\mathbf{Top^w}}((X,x_0),({\cal X}_{G},*_G))$.
Firstly $\theta$ is well-defined because of the well-definition of
${\cal X}_{-}$, also it is one to one and onto due to the fact
that $\bar{\pi}$ and ${\cal X}_{-}$ are inverse to each other.
Also the map $\theta$, as a bijection, is natural in each
variable; that is the following diagrams commute:
$$\begin{array}{ccccccc}
{\mathrm Hom}_{\pi({\mathbf{Top^w}})}(\bar{\pi}(X_1),G_1)
&\stackrel{(\bar{\pi}(f))^*}{\longrightarrow}
& {\mathrm Hom}_{\pi({\mathbf{Top^w}})}(\bar{\pi}(X_2),G_1)\\
\downarrow \theta & & \downarrow \theta\\
{\mathrm Hom}_{\mathbf{sqTop^w}}(X_1,{\cal X}_{G_1})
&\stackrel{(f)^*}{\longrightarrow}& {\mathrm
Hom}_{\mathbf{sqTop^w}}(X_2,{\cal X}_{G_1})
\end{array}$$
and
$$\begin{array}{ccccccc}
{\mathrm Hom}_{\pi({\mathbf{Top^w}})}(\bar{\pi}(X_1),G_1)
&\stackrel{(g_*)_*}{\longrightarrow}
& {\mathrm Hom}_{\pi({\mathbf{Top^w}})}(\bar{\pi}(X_1),G_2)\\
\downarrow \theta & & \downarrow \theta\\
{\mathrm Hom}_{\mathbf{sqTop^w}}(X_1,{\cal X}_{G_1})
&\stackrel{({\cal X}_{g_*})_*}{\longrightarrow}& {\mathrm
Hom}_{\mathbf{sqTop^w}}(X_1,{\cal X}_{G_2})
\end{array}$$
for all $f:X_2\rightarrow X_1$ in ${\mathbf{sqTop^w}}$ and
$g_*:G_1\rightarrow G_2$ in $\pi(\mathbf{Top^w})$. This note ends
the proof.
\end{proof}
The above theorem may be the most essential fact in this paper so
that the existence of the right adjoint for a functor can imply
some known manner such as preserving direct limits, if any, and
keeping the right exactness of an exact sequence, in a certain
sense. We end with the following result.
\begin{thm} For every direct
system $\{(X_i,x_i),i\in I\}$ of spaces in a subcategory
$\mathbf{sqTop^w}$ whose direct limit belongs to the category,
we have the isomorphism
$\pi(\displaystyle{\lim_{\rightarrow}X_i)}\cong\displaystyle{\lim_{\rightarrow}\pi(X_i)}$.
\end{thm}
\begin{proof} The functor $\bar{\pi}$ which has
a right adjoint ${\cal X}_{-}$, should preserve the direct limit
[7], that is the above isomorphism and so the result holds.
\end{proof}
\section{A Theorem of Van Kampen Type for Join Spaces}
In this section, we are in a position to present one of the main
result of the paper which is a theorem of Van Kampen type theorem
for wild pseudo Peano continuum spaces. Note that the well-known
Van Kampen Theorem asserts that the fundamental group of a join
of first countable semilocally simply connected spaces is
isomorphic to the free product of their fundamental groups [4].
However, the conditions presented in the following theorem are
completely different from (even though opposite to) the
conditions in Van Kampen Theorem.
\begin{lem} The join of any
finite family of spaces in $\mathbf{sqTop^w}$ is also in this
category.
\end{lem}
\begin{proof} Let $\{X_i\}_{i\in I}$ be a finite
family of spaces in $\mathbf{sqTop^w}$ and
$(X,x_0)=\bigvee(X_i,x_i)$ be the join space of this family.
First, we recall the definition of $\mathbf{sqTop^w}$ and the
note after it, also Lemma 3.2 which asserts that the join space
of pseudo Peano continuum spaces is also pseudo Peano continuum.
So to complete the proof, it is sufficient to prove the space $X$
is also wild. By the contrary, suppose there exists a point $x\in
X$ and a neighborhood $U_x$ of it so that the homomorphism
$\pi(j):\pi(U_x)\rightarrow \pi(X)$, induced by inclusion, is
trivial. Now if we composite the inclusion map $l_i:U_x\cap
X_i\rightarrow X_i$ and the collapsing map $p_i:X\rightarrow
X_i$, we obtained the inclusion $p_i\circ j\circ l_i: U_x\cap
X_i\rightarrow X_i$ which induces the trivial homomorphism
$\pi(p_i\circ j\circ l_i)=\pi(p_i)\pi(j)\pi(l_i):\pi(U_x\cap
X_i)\rightarrow \pi(X_i)$ which contradicts to supposition of
$X_i$ to be wild space.
\end{proof}
\begin{thm} (fundamental groups
 of join spaces in ${\mathbf Top^w}$). The
 fundamental group of the join of a finite family of wild pseudo Peano
continuum spaces is isomorphic to the coproduct of their
fundamental groups in $\pi(\mathbf{Top^w})$.
\end{thm}
\begin{proof} Firstly, we recall that the join
space $(X,x)=\bigvee(X_i,x_i)$ as coproduct of the finite family
$\{X_i,i\in I\}$ is the direct limit of a special direct system
explained in [7,8]. Now, by considering a suitable subcategory
$\mathbf{sqTop^w}$ contains all the pointed space $(X_i,x_i)$ and
using Lemma 5.1 and Theorem 4.2, the fundamental group of
$(X,x)=\bigvee(X_i,x_i)$ must be the direct limit of the induced
corresponding direct system, which plays the role of the coproduct
of $\{\pi(X_i)\}_{i\in I}$ in the category
$\pi($${\mathbf{Top^w}})$. Since $\bigvee(X_i,x_i)$ is a wild
pseudo Peano continuum space and
$\pi(\bigvee(X_i,x_i))\cong\displaystyle{\lim_{\rightarrow}\pi(X_i,x_i)}$,
the above coproduct in the category $\pi($${\mathbf{Top^w}})$ does exist.
\end{proof}
As we see, the results of this section are deduced essentially by
using the functorial property of the fundamental group on the
special category which we constructed. Now attending to the
functor ${\cal X}_{-}$ as a right adjoint to $\bar{\pi}$, we will
get a result which is indeed another proof for Theorem 2.3 [1,
Corollary 5.2]. Details are offered in the next section.
\section{Fundamental Groups of Quotient Spaces}
In this final section we are going to find a relation between
fundamental groups of a quotient space $X/A$, the space $X$ and
its subspace $A$, where $X/A$, $X$, and $A$ belong to the
category ${\mathbf{Top^w}}$. In order to do this, we need to
consider cokernels as direct limits in $\mathbf{sqTop^w}$, but
this category has no zero object. In order to compensate this
default we consider the following two categories:

\hspace{-0.65cm}\textbf{{\bf The Category $\mathbf{sqTop_0^w}$;}
} is obtained by $\mathbf{sqTop^w}$ by adding the point $\{*\}$
and all morphisms, the constant map $0^X:(X,x_0)\rightarrow
\{*\}$, and the map $0_Y:\{*\}\rightarrow (Y,y_0)$ taking $*$ to
the point $y_0$, to the category $\mathbf{sqTop^w}$.

\hspace{-0.65cm}\textbf{{\bf The Category
$\mathbf{\pi(Top_0^w)}$;}} is obtained by $\mathbf{\pi(Top_0^w)}$
by adding the trivial group $\{e\}$ and all morphisms, the
trivial homomorphism $0^{\pi(X)}:\pi(X,x_0)\rightarrow \{e\}$ and
the homomorphism $0_{\pi(Y)}:\{e\}\rightarrow \pi(Y,y_0)$ mapping
$e$ to the identity $e$, to the category $\pi(\mathbf{Top^w})$.

Moreover, we should consider two functors
$\bar{\pi}_0:\mathbf{sqTop_0^w}\rightarrow \pi(\mathbf{Top_0^w})$
and ${\cal X}_{0 -}:\pi(\mathbf{Top_0^w})\rightarrow
\mathbf{sqTop_0^w}$ which are natural extensions of $\bar{\pi}$
and ${\cal X}_{-}$, and correspond added objects $\{*\}$ and
$\{e\}$ to each other, such as:
$$\bar{\pi}_0: \{*\}\longmapsto \{e\}\ \ \ \ {\mathrm and}
\ \ \ {\cal X}_{0 -}:\{e\}\longmapsto \{*\}.$$

Similar to Theorem 5.1, one can easily verify that
$(\bar{\pi}_0,{\cal X}_{0 -})$ is an adjoint pair, and hence
$\bar{\pi}_0$ preserves direct limits in $\mathbf{sqTop_0^w}$, and
${\cal X}_{0 -}$ preserves inverse limits in
$\pi(\mathbf{Top_0^w})$.

Now, consider a pointed space $(X,x_0)$ and a subspace $A$
containing $x_0$ such that $(X,x_0)$, $(A,x_0)$, and $(X/A,*)$
belong to $\mathbf{sqTop_0^w}$. Note that we can choose a suitable
fixed point $*_{\pi(X)}$ in ${\cal X}_{\pi(X)}$ and $*_{\pi(A)}$
in ${\cal X}_{\pi(A)}$ such that $\varphi(*_{\pi(X)})=x_0$ and
$\psi (*_{\pi(A)})=x_0$, where $\psi:{\cal X}_{\pi(A)}\rightarrow
A$ is the homeomorphism.

We can consider the quotient space $(X/A,*)$ as a pushout of the
following diagram:

$$\begin{array}{ccccccc}
(A,x_0) &\stackrel{j}{\longrightarrow}
& (X,x_0)\\
\downarrow 0^A & & \downarrow p\\
\{*\} &\stackrel{nat}{\longrightarrow}& (X/A,*),
\end{array}$$
where $j$ and $p$ are inclusion and quotient maps, respectively.
The map $nat$ also naturally corresponds $*$ to the point $*$.

Therefore the pointed space $(X/A,*)$ can be considered as a
direct limit in $\mathbf{sqTop_0^w}$, and hence $\pi(X/A,*)$ is
the pushout of the following diagram:
$$\begin{array}{ccccccc}
\pi(A,x_0) &\stackrel{j_*}{\longrightarrow}
& \pi(X,x_0)\\
\downarrow 0^{\pi(A)} & & \downarrow p_*\\
\{e\} &\stackrel{nat_*}{\longrightarrow}& \pi(X/A,*),
\end{array}$$
where $j_*$ and $p_*$ are the induced map $\pi(j)$ and $\pi(p)$,
respectively. The homomorphism $nat_*$ is also mapping $e$ to the
identity of $\pi(X/A,*)$. Also, we know that
$\pi(X,x_0)/{j_*(\pi(A,x_0))}^{\pi(X,x_0)}$ is the
pushout of the above diagram in the category of groups. Hence
there exists a unique homomorphism
$$\psi:\frac{\pi(X,x_0)}{{j_*(\pi(A,x_0))}^{\pi(X,x_0)}}\rightarrow
\pi(X/A,*)$$
such that the following diagram commutes:
\[\begin{picture}(240,70)
\put(48,55){$\pi(X,x_0)$} \put(152,60){$\pi(X/A,*)$}
\put(1,1){$\pi(X,x_0)/{j_*(\pi(A,x_0))}^{\pi(X,x_0)}.$}
\put(98,59){\vector(1,0){45}} \put(140,8){\vector(1,1){45}}
\put(75,51){\vector(0,-1){25}} \put(117,62){\small $p_*$}
\put(145,26){\small $\psi$} \put(57,35){\small $nat$}.
\end{picture}\]

\end{document}